\documentclass{amsart}
\usepackage{amssymb,amsmath}
\usepackage{graphicx}
\usepackage{xcolor}

\theoremstyle{plain}
\newtheorem{thm}{Theorem}[section]

\newtheorem{cor}[thm]{Corollary}
\newtheorem{lem}[thm]{Lemma}
\newtheorem{prop}[thm]{Proposition}

\newtheorem{defn}[thm]{Definition}

\newtheorem{rem}[thm]{Remark}
\newtheorem{ex}[thm]{Example}

\numberwithin{equation}{section}

\newcommand{\ra}{\rightarrow}
\newcommand{\Ra}{\Rightarrow}

\newcommand{\Lra}{\Leftrightarrow}

\begin{document}

\title[Products of random walks]{Products of random walks on finite groups with moderate growth}

\author[G.-Y. Chen]{Guan-Yu Chen$^1$}

\author[T. Kumagai]{Takashi Kumagai$^2$}

\address{$^1$Dept. of Appl. Math., National Chiao Tung University, Hsinchu 300, Taiwan}
\email{gychen@math.nctu.edu.tw}

\address{$^2$RIMS, Kyoto University, Kyoto 606-8502, Japan}\email{kumagai@kurims.kyoto-u.ac.jp}



\begin{abstract}
In this article, we consider products of random walks on finite groups with moderate growth and discuss their cutoffs in the total variation. Based on several comparison techniques, we are able to identify the total variation cutoff of discrete time lazy random walks with the Hellinger distance cutoff of continuous time random walks. Along with the cutoff criterion for Laplace transforms, we derive a series of equivalent conditions on the existence of cutoffs, including the existence of pre-cutoffs, Peres' product condition and a formula generated by the graph diameters. For illustration, we consider products of Heisenberg groups and randomized products of finite cycles.
\end{abstract}

\maketitle

\section{Introduction}

Let $G$ be a finite group equipped with a probability $Q$. A random walk on $G$ driven by $Q$ is a discrete time Markov chain with state space $G$ and transition matrix $K$ given by $K(x,y)=Q(x^{-1}y)$. If $K$ is irreducible, then the stationary distribution $U$ is uniform on $G$. For simplicity, we write the triple $(G,Q,U)$ for such a random walk. Here, $Q$ is called symmetric if $Q(x)=Q(x^{-1})$ for all $x\in G$ and, in this case, $(G,Q,U)$ is named a symmetric random walk. Note that if $Q$ is symmetric, then $K$ is reversible. To study the convergence of $(G,Q,U)$, we consider the total variation and its corresponding mixing time, which are defined respectively by
\begin{equation}\label{eq-tv}
 d_{\textnormal{\tiny TV}}(x,m):=\max_{E\subset G}\{Q^{(m)}(x^{-1}E)-U(E)\},
\end{equation}
and
\begin{equation}\label{eq-tvmix}
 T_{\textnormal{\tiny TV}}(x,\epsilon):=\min\{m\ge 0|d_{\textnormal{\tiny TV}}(x,m)\le\epsilon\},
\end{equation}
where $x^{-1}E=\{x^{-1}y|y\in E\}$ and $Q^{(m)}$ is the $m$-fold convolution product $Q*\cdots*Q$ with
\[
 f*g(x)=\sum_{z\in G}f(z)g(z^{-1}x).
\]
As the total variation and, thus, its mixing time are constant in $x$, we shall write $d_{\textnormal{\tiny TV}}(m)$ and $T_{\textnormal{\tiny TV}}(\epsilon)$ for short.

A subset $E\subset G$ is called symmetric, if $x\in E$ implies $x^{-1}\in E$, and is called a generating set of $G$, if $E^n:=\{x_1x_2\cdots x_n|x_i\in E,\,\forall 1\le i\le n\}=G$ for some $n>0$. We write $(G,E)$ for the Cayley graph with vertex set $G$ and edge set $\{(x,xy)|x\in G,\,y\in E\}$ and define its volume growth function and diameter by
\[
 V(m)=|E^m|,\quad \rho=\min\{m\ge 1|V(m)=|G|\}.
\]
A group $G$ is said to have $(A,d)$-moderate growth with respect to a generating set $E$ if
\[
 \frac{V(m)}{V(\rho)}\ge \frac{1}{A}\left(\frac{m}{\rho}\right)^d,\quad\forall 1\le m\le \rho.
\]
The following are some typical groups with moderate growth.

{\it Example 1:} When $G=\mathbb{Z}_n$ and $E=\{0,\pm 1\}$, the graph $(G,E)$ has diameter $\rho=\lfloor n/2\rfloor$ and $G$ has $(1,1)$-moderate growth w.r.t. $E$ for $n\ge 2$.

{\it Example 2:} When $G=\mathbb{Z}_n$ and $E=\{0,\pm 1,\pm\lfloor\sqrt{n}\rfloor\}$, the diameter $\rho$ is of order $\sqrt{n}$ and $G$ has $(1,2)$-moderate growth w.r.t. $E$ for $n\ge 2$.

{\it Example 3:} When $G$ is the Heisenberg group mod $n+2$, which is the set of $3\times 3$ matrices of the form
\begin{equation}\label{eq-Heisenberg}
 \left(\begin{array}{ccc}1&i&k\\0&1&j\\0&0&1\end{array}\right),\quad\forall i,j,k\in\mathbb{Z}_{n+2},
\end{equation}
and $E$ contains the following matrices
\begin{equation}\label{eq-HeisenbergE}
 I,\left(\begin{array}{ccc}1&1&0\\0&1&0\\0&0&1\end{array}\right),
 \left(\begin{array}{ccc}1&-1&0\\0&1&0\\0&0&1\end{array}\right),
 \left(\begin{array}{ccc}1&0&0\\0&1&1\\0&0&1\end{array}\right),
 \left(\begin{array}{ccc}1&0&0\\0&1&-1\\0&0&1\end{array}\right),
\end{equation}
it was proved in \cite[Lemma 4.1]{DS94} that $(G,E)$ has diameter $n+1\le\rho\le n+4$ and $G$ has $(48,3)$-moderate growth w.r.t. $E$ for $n\ge 1$.

Throughout this article, we will simply write $id$ for the identity of any group. In \cite{DS94}, Diaconis and Saloff-Coste considered random walks on finite groups with moderate growth and achieved the following proposition.

\begin{prop}[Theorem 3.1 in \cite{DS94}]\label{p-moderate}
Let $(G,Q,U)$ be a symmetric random walk on a finite group and $E$ be a symmetric generating set of $G$ containing $id$. Assume that $G$ has $(A,d)$-moderate growth with respect to $E$ and $\eta=\min\{Q(x)|x\in E\}>0$. Then, there is $C_1=C_1(A,d)>0$ such that
\begin{equation}\label{eq-moderateub}
 d_{\textnormal{\tiny TV}}(m)\le C_1e^{-\eta m/\rho^2},\quad\forall m\ge 0,
\end{equation}
where $\rho$ is the diameter of $(G,E)$. If it is assumed further that $Q$ is supported on $E$ and that $\rho\ge A2^{2d+2}$, then there is $C_2=C_2(A,d)>0$ such that
\begin{equation}\label{eq-moderatelb}
 d_{\textnormal{\tiny TV}}(m)\ge \frac{1}{2}e^{-C_2m/\rho^2},\quad\forall m\ge 0.
\end{equation}
\end{prop}
In fact, the authors of \cite{DS94} obtain $C_1=A^{1/2}2^{d(d+3)/4}$ and $C_2=A^22^{4d+2}$. This means that the bounds in (\ref{eq-moderateub})-(\ref{eq-moderatelb}) are far from comparable when $A$ or $d$ is large.

\smallskip

We now consider product chains. Let $(G_i,Q_i,U_i)_{i=1}^n$ be irreducible random walks on finite groups and $(p_1,...,p_n)$ be a probability vector. Define
\begin{equation}\label{eq-proddis}
 G=G_1\times\cdots\times G_n,\quad U=U_1\times\cdots\times U_n,\quad Q(x)=\sum_{i=1}^np_iQ_i(x_i),
\end{equation}
for $x=(x_1,..,x_n)\in G$. Here, $(G,Q,U)$ is called the product of $(G_i,Q_i,U_i)_{i=1}^n$ with respect to the probability vector $(p_1,...,p_n)$. Note that if $E_i$ is the support of $Q_i$ and contains $id$, then $Q$ is supported on $E=\check{E}_1\cup\cdots\cup\check{E}_n$, where $\check{E}_i=\{x=(x_1,...,x_n)\in G|x_i\in E_i,x_j=id,\forall j\ne i\}$. Further, if $E_i$ is a symmetric generating set of $G_i$ and $\rho_i$ is the diameter of $(G_i,E_i)$, then $E$ is a symmetric generating set of $G$ and the diameter $\rho$ of $(G,E)$ satisfies $\rho=\rho_1+\cdots+\rho_n$. To see the moderate growth of direct products of groups, let $E_i,E$ be as before and assume that $G_i$ has $(A_i,d_i)$-moderate growth w.r.t. $E_i$. As $G$ is a finite group and $E$ generates $G$, there are always positive constants $A,d$ such that $G$ has $(A,d)$-moderate growth w.r.t. $E$. However, the relation between $(A,d)$ and $(A_i,d_i)_{i=1}^n$ could be complicated and, in general, $A$ or $d$ can be very large when $n$ grows. (For instance, consider $G_i=\mathbb{Z}_N$ and $E_i=\{0,\pm 1\}$ for $1\le i\le n$. It is mentioned earlier that $G_i$ has $(1,1)$ moderate growth w.r.t. $E_i$. In some combinatoric computations, one may show that, under the assumption of $N\ge n^2$, $G$ has $(A,d)$ moderate growth with $A=(1-1/N)^{-2}$ and $d=n$.) Consequently, (\ref{eq-moderateub}) and (\ref{eq-moderatelb}) might not be sharp enough to provide efficient bounds on the total variation even if the prerequisites, $\eta>0$ and $\rho\le A2^{2d+2}$, are fulfilled. To proceed the analysis of product chains, as the total variation mixing times are comparable between $(G,Q,U)$ and its associated continuous time walk, see e.g. \cite{CSal13-1}, it is more convenient, as is discussed below, to consider the continuous time chain rather than the discrete time one.

Given a random walk $(G,Q,U)$, we associate it with a continuous time random walk $(G,H_t,U)$, where $H_t=e^{t(K-I)}$ and $K$ is the transition matrix given by $Q$. One realization of $(G,H_t,U)$ is to change the constant waiting times of $(G,Q,U)$ into an i.i.d. sequence of exponential random variables. Note that, if $(G,Q,U)$ is the product of $(G_i,Q_i,U_i)_{i=1}^n$ with respect to the probability vector $(p_1,...,p_n)$ and $(G_i,H_{i,t},U_i)$ is the continuous time random walk associated with $(G_i,Q_i,U_i)$, then
\begin{equation}\label{eq-prodcts}
 H_t=H_{1,p_1t}\otimes\cdots\otimes H_{n,p_nt},
\end{equation}
where $A\otimes B$ denotes the tensor product of matrices $A$ and $B$. In general, $K^m$ does not have the form of (\ref{eq-prodcts}). Through (\ref{eq-prodcts}), one may study $H_t$ via $(H_{i,t})_{i=1}^n$ but, unfortunately, there lacks an efficient expression of the total variation of $(G,H_t,U)$ in terms of the total variations of $(G_i,H_{i,t},U_i)_{i=1}^n$.

In \cite{CK16}, two inequalities were used to compare the total variation and the Hellinger distance and this leads to a different way to analyze their mixing times. In detail, the Hellinger distance of $(G,Q,U)$ is defined by
\begin{equation}\label{eq-hd}
 d_H(x,m):=\left(\frac{1}{2}\sum_{y\in G}\left(\sqrt{K^m(x,y)}-\sqrt{U(y)}\right)^2\right)^{1/2},
\end{equation}
while the Hellinger distance of $(G,H_t,U)$ is defined by replacing $K^m$ with $H_t$ in (\ref{eq-hd}) and denoted by $d_H^{(c)}(x,t)$ in avoidance of confusion. As before, we will write $d_H(m)$ (resp. $d_H^{(c)}(t)$) for short since $d_H(x,m)$ (resp. $d_H^{(c)}(x,t)$) is constant in $x$. In the above setting, Equation (1.3) in \cite{CK16} says that
\begin{equation}\label{eq-tvhdcomp}
 1-\sqrt{1-d_{\textnormal{\tiny TV}}^2(m)}\le d_H^2(m)\le d_{\textnormal{\tiny TV}}(m),
\end{equation}
and also hold in the continuous time case. In the Hellinger distance, if $(G,Q,U)$ is the product of $(G_i,Q_i,U_i)_{i=1}^n$ with respect to the probability vector $(p_1,...,p_n)$, then the Hellinger distances, $d_H^{(c)}$ and $d_{i,H}^{(c)}$, of $(G,H_t,U)$ and $(G_i,H_{i,t},U_i)$ satisfy
\begin{equation}\label{eq-prodhd}
 d_H^{(c)}(t)^2=1-\prod_{i=1}^n\left(1-d_{i,H}^{(c)}(p_it)^2\right).
\end{equation}
Such an equality is derived from (\ref{eq-prodcts}) but not applicable to the discrete time case. See p.365 in \cite{Sh96} or Lemma 2.3 in \cite{CK16} for a proof of (\ref{eq-tvhdcomp}) and see \cite{CK16} for more comparisons of mixing times of product chains.

In this article, we focus on the cutoff phenomenon, or briefly cutoff, for products of random walks on finite groups with moderate growth. The cutoff of Markov chains was introduced by Aldous and Diaconis in early 1980s in order to catch up the
phase transition of the mixing time. To see a definition, let $\mathcal{F}=(G_n,Q_n,U_n)_{n=1}^\infty$ be a family of random walks on finite groups. For $n\ge 1$, let $d_{n,\textnormal{\tiny TV}}$ and $T_{n,\textnormal{\tiny TV}}$ be the total variation and corresponding mixing time of the $n$th chain in $\mathcal{F}$. Assume that $T_{n,\textnormal{\tiny TV}}(\epsilon_0)\ra\infty$ for some $\epsilon_0\in(0,1)$. The family $\mathcal{F}$ is said to present a cutoff in the total variation if
\begin{equation}\label{eq-defcutoff1}
 \lim_{n\ra\infty}\frac{T_{n,\textnormal{\tiny TV}}(\epsilon)}{T_{n,\textnormal{\tiny TV}}(\delta)}=1,\quad\forall \epsilon,\delta\in(0,1),
\end{equation}
or, equivalently
(see \cite[Proposition 2.4]{CSal08}),
there is a sequence of positive reals $(t_n)_{n=1}^\infty$ such that
\begin{equation}\label{eq-defcutoff2}
 \lim_{n\ra\infty}d_{n,\textnormal{\tiny TV}}(\lceil at_n\rceil)=0\quad\forall a>1,\quad\lim_{n\ra\infty}d_{n,\textnormal{\tiny TV}}(\lfloor at_n\rfloor)=1,\quad\forall a\in(0,1).
\end{equation}
When a cutoff exists, the sequence $(t_n)_{n=1}^\infty$, or briefly $t_n$, in (\ref{eq-defcutoff2}) is called a cutoff time. By (\ref{eq-defcutoff1}), it is easy to see that $T_{n,\textnormal{\tiny TV}}(\epsilon)$ can be selected as cutoff time for any $\epsilon\in(0,1)$. In the continuous time case, we write $\mathcal{F}_c=(G_n,H_{n,t},U_n)_{n=1}^\infty$ for the family of continuous time Markov chains associated with $\mathcal{F}$ and use $d_{n,\textnormal{\tiny TV}}^{(c)}$ and $T_{n,\textnormal{\tiny TV}}^{(c)}$ to denote the total variation and its mixing time of the $n$th chain in $\mathcal{F}_c$. The total variation cutoff of $\mathcal{F}_c$ is defined in the same way through (\ref{eq-defcutoff1}) or (\ref{eq-defcutoff2}) under the replacement of $T_{n,\textnormal{\tiny TV}},d_{n,\textnormal{\tiny TV}}$ with $T_{n,\textnormal{\tiny TV}}^{(c)},d_{n,\textnormal{\tiny TV}}^{(c)}$ and the removal of $\lceil\cdot\rceil,\lfloor\cdot\rfloor$ but without the prerequisite of $T_{n,\textnormal{\tiny TV}}^{(c)}(\epsilon_0)\ra\infty$. All above is also applicable to the Hellinger distance. We refer readers to \cite{D88,SC04} for more discussions on cutoffs for random walks on finite groups.

Cutoffs in the total variation and in the Hellinger distance were proved to be equivalent in \cite{CK16} via (\ref{eq-tvhdcomp}). Since no similar formula to (\ref{eq-prodhd}) is available for the total variation or for the discrete time case, it is straightforward to consider the cutoff in the Hellinger distance for families of continuous time product chains. For finite groups with moderate growth, we obtain a continuous time variant of Proposition \ref{p-moderate} in Proposition \ref{p-moderatehd} with a refined assumption on the lower bound (from $\rho\ge A2^{2d+2}$ to $\rho\ge 4$). Through (\ref{eq-prodhd}), the Hellinger distances of product chains can be expressed in a form related to sums of exponential functions. By regarding those sums as Laplace transforms, a criterion in \cite{CHS16} was proposed to determine the cutoff and to characterize the cutoff time. Table \ref{tab-scheme} is the conclusive scheme of all above discussions.\vskip2mm

\begin{table}[h]
\caption{A scheme to analyze cutoffs}\label{tab-scheme}
\begin{center}
\begin{tabular}{l}
{\bf Total variation cutoffs for discrete time Markov chains}\\
\qquad$\Updownarrow$ \quad(Proposition 1.1 in \cite{CSal13-1} holds for lazy random walks)\\
{\bf Total variation cutoffs for continuous time Markov chains}\\
\qquad$\Updownarrow$ \quad(Theorem 1.1 in \cite{CK16} holds for any Markov chain)\\
{\bf Hellinger distance cutoffs for continuous time Markov chains}\\
\qquad$\Updownarrow$ \quad(Proposition \ref{p-moderatehd} holds for groups with moderate growth)\\
{\bf Cutoffs for Laplace transforms}\\
\qquad$\Updownarrow$ \quad(Theorem 2.4 in \cite{CHS16} holds for reversible Markov chains)\\
{\bf Precise condition on cutoffs}
\end{tabular}
\end{center}
\end{table}

The aim of this paper is to establish necessary and sufficient conditions for cutoff
for products of random walks on finite groups with moderate growth, and apply them to
stimulating examples. In the first main theorem (Theorem \ref{p-prodmoderate}), we give
various equivalent conditions for cutoff in our framework. It should be noted that in this framework
the cutoff is equivalent to a weaker concept, called the pre-cutoff (note that
such an equivalence generally fails; see \cite{L15}).
Moreover, one equivalent condition in Theorem \ref{p-prodmoderate} is consistent with Peres' conjecture (see Remark \ref{r-Peres}\,(3)), while another is simply determined by the graph diameters and a  sequence $\mathcal{P}$ given below. In the second main theorem (Theorem \ref{t-prodmoderate}),
we apply Theorem \ref{p-prodmoderate} to the specific type of products introduced in \cite{CK16} and derive more concrete conditions on their cutoffs.
To illustrate our results, let us consider products of random walks on Heisenberg groups and randomized products of random walks on finite cycles. Let $\mathcal{G}=(G_n,Q_n,U_n)_{n=1}^\infty$ be a family of random walks on finite groups and $\mathcal{P}=(p_n)_{n=1}^\infty$ be a sequence of positive reals. For $n\ge 1$, let $q_n=\sum_{i=1}^np_i$ and write $\mathcal{G}^\mathcal{P}$ for the family of which $n$th random walk is the product of $(G_i,Q_i,U_i)_{i=1}^n$ according to the probability vector $(p_1/q_n,...,p_n/q_n)$. Then the following hold.

\begin{prop}\label{t-Heisenberg}
Let $\mathcal{G}=(G_n,Q_n,U_n)_{n=1}^\infty$, where $G_n$ is the Heisenberg group in (\ref{eq-Heisenberg}), $Q_n$ is the probability uniformly supported on the set $E_n$ in (\ref{eq-HeisenbergE}) and $p_n=n^2\exp\{-n^{\gamma}\}$ with $\gamma>0$. Then, $\mathcal{G}^\mathcal{P}$ has a total variation cutoff if and only if $0<\gamma<1$.
\end{prop}

\begin{prop}\label{t-rp}
Consider the family $\mathcal{G}=(\mathbb{Z}_{n+2},Q_n,U_n)_{n=1}^\infty$, where $Q_n(0)=1/2$ and $Q_n(1)=Q_n(-1)=1/4$. Let $(X_n)_{n=1}^\infty$ be i.i.d. positive random variables and $\mathcal{P}=(p_n)_{n=1}^\infty$ be a random sequence given by $(X_n)_{n=1}^\infty$.
\begin{itemize}
\item[(1)] Suppose $p_n=(X_1+\cdots+X_n)^\gamma$ with $\gamma>0$ and $X_1$ has a finite expectation. For $\gamma\in(0,2]$, $\mathcal{G}^\mathcal{P}$ has a total variation cutoff with probability $1$. For $\gamma>2$, $\mathcal{G}^\mathcal{P}$ has no total variation cutoff with probability $1$.

\item[(2)] Suppose $p_n=X_1\times\cdots\times X_n$ and $\log X_1$ has a positive finite expectation. Then, $\mathcal{G}^\mathcal{P}$ has no total variation cutoff with probability $1$.
\end{itemize}
\end{prop}
Proposition \ref{t-Heisenberg} is an immediate result of Corollary \ref{c-prodmoderate}, which shows a phase transition of cutoffs at $\gamma=1$. Proposition \ref{t-rp} is of its own interest and discussed in detail in Subsection \ref{ss-ex}.

\smallskip

The remaining of this paper is organized in the following way. In Section \ref{s-trta}, we introduce the core results of the paper. The main theorems are given in Subsections \ref{ss-tmt} and \ref{ss-tmr}. As an example, we consider the randomized product
and discuss its cutoff in Subsection \ref{ss-ex}. Section \ref{s-ctf} is dedicated to the construction of framework in Section \ref{s-trta}. In Subsection \ref{ss-rots}, we review and develop some theoretical results that are crucial to the equivalences in Table \ref{tab-scheme}, while in Subsection \ref{ss-dof}, Theorem \ref{p-prodmoderate} is proved in detail. To make this paper more readable, we address those minor and involved results in the appendix.

We end the introduction by quoting the following notations. Let $x,y\in\mathbb{R}$ and $a_n,b_n$ be sequences of positive reals. We write $x\vee y$ and $x\wedge y$ for the maximum and minimum of $x,y$. When $a_n/b_n$ is bounded, we write $a_n=O(b_n)$; when $a_n/b_n\ra 0$, we write $a_n=o(b_n)$. In the case of $a_n=O(b_n)$ and $b_n=O(a_n)$, we simply say $a_n\asymp b_n$. If $a_n/b_n\ra 1$, we write $a_n\sim b_n$. In computations, $O(a_n)$ and $o(b_n)$ denote two sequences $c_n$ and $d_n$ satisfying $|c_n/a_n|=O(1)$ and $|d_n/b_n|=o(1)$.

\section{Main theorems and applications}\label{s-trta}

In this section, we will introduce our main results in the general setting and discuss their applications, including Proposition \ref{t-Heisenberg}.

\subsection{Framework and main theorem}\label{ss-tmt}

In this subsection, we introduce the theoretical framework and one of the main theorems
in this article. First, let us consider a concept weaker than cutoff.

\begin{defn}\label{d-precutoff}
Let $\mathcal{F}=(G_n,Q_n,U_n)_{n=1}^\infty$ be a family of random walks on finite groups and $d_{n,\textnormal{\tiny TV}}$ be the total variation of $(G_n,Q_n,U_n)$. $\mathcal{F}$ is said to present a pre-cutoff in the total variation if there are $B>A>0$ and a sequence $t_n>0$ such that
\begin{equation}\label{eq-precutoff}
 \lim_{n\ra\infty}d_{n,\textnormal{\tiny TV}}(\lceil Bt_n\rceil)=0,\quad \liminf_{n\ra\infty}d_{n,\textnormal{\tiny TV}}(\lfloor At_n\rfloor)>0.
\end{equation}
\end{defn}

\begin{rem}
{\rm (1)} The removal of $\lceil\cdot\rceil,\lfloor\cdot\rfloor$ provides the pre-cutoff for $\mathcal{F}_c$ and the replacement of $d_{n,\textnormal{\tiny TV}}$ with $d_{n,H}$ yields the pre-cutoff in the Hellinger distance. When $t_n\ra\infty$, the pre-cutoff in Definition \ref{d-precutoff} is equivalent to
\begin{equation}\label{eq-precutoffmix}
 \lim_{\epsilon\ra 0}\limsup_{n\ra\infty}\frac{T_{n,\textnormal{\tiny TV}}(\epsilon)}{T_{n,\textnormal{\tiny TV}}(\epsilon_0-\epsilon)}<\infty,
\end{equation}
for some $\epsilon_0\in(0,1)$. Such an equivalence also holds for $\mathcal{F}_c$ without the prerequisite of $t_n\ra\infty$.\\
{\rm (2)} Definition \ref{d-precutoff} was introduced for the purpose of studying the mixing times and cutoffs for families of Markov chains. Readers are referred to \cite{SC97,CSal08} for more discussions on this subject. It is worthwhile to note that there is indeed another (stronger) variant of pre-cutoff in \cite[Chapter 18]{LPW08}, of which definition is similar to (\ref{eq-precutoff}) except the replacement of the second limit by
\[
 \lim_{n\ra\infty}d_{n,\textnormal{\tiny TV}}(\lfloor At_n\rfloor)=1.
\]
When $t_n\ra\infty$, the equivalence of the pre-cutoff and (\ref{eq-precutoffmix}) also holds for such a variant with $\epsilon_0=1$. We would like to emphasize that Theorem \ref{p-prodmoderate} (discussed later) remains true when the pre-cutoff refers to the stronger one.
\end{rem}

In the following, we consider a rather general setting than Proposition \ref{t-Heisenberg}. Let $(k_n)_{n=1}^\infty$ be a sequence of positive integers and
\begin{equation}\label{eq-triarray}
 \mathcal{F}=\{(G_{n,i},Q_{n,i},U_{n,i})|1\le i\le k_n,n\ge 1\},\quad\mathcal{P}=\{p_{n,i}|1\le i\le k_n,n\ge 1\},
\end{equation}
where $(G_{n,i},Q_{n,i},U_{n,i})$ is a random walk on a finite group and $p_{n,i}>0$. We write $\mathcal{F}^\mathcal{P}$ for the family $(G_n,Q_n,U_n)_{n=1}^\infty$, where $(G_n,Q_n,U_n)$ is the product of $(G_{n,i},Q_{n,i},U_{n,i})_{i=1}^{k_n}$ according to the probability vector $(p_{n,i}/q_n)_{i=1}^{k_n}$ and $q_n=p_{n,1}+\cdots+p_{n,k_n}$. As before, we use $d_{n,\textnormal{\tiny TV}}^{(c)},T_{n,\textnormal{\tiny TV}}^{(c)}$ to denote the total variation and its mixing time of the $n$th chain in $\mathcal{F}^\mathcal{P}_c$. Along with these notations, we are ready to state the
first main theorem of this article.
\begin{thm}\label{p-prodmoderate}
Refer to the triangular arrays in (\ref{eq-triarray}). Let $E_{n,i}$ be the support of $Q_{n,i}$ and $\rho_{n,i}$ be the diameter of $(G_{n,i},E_{n,i})$. Assume that $Q_{n,i}$ is symmetric, $\inf\{Q_{n,i}(x)|x\in E_{n,i},\,1\le i\le k_n\,n\ge 1\}>0$ and $G_{n,i}$ has $(A,d)$-moderate growth with respect to $E_{n,i}$ for all $n,i$. Assume further that $\rho_{n,1}\ge 4$ for $n$ large enough and there are $C>1$ and $\ell_{n,i}>0$ satisfying $\ell_{n,i}\le \ell_{n,i+1}$ such that $\ell_{n,i}/C\le p_{n,i}/(q_n\rho_{n,i}^2)\le C\ell_{n,i}$ for all $n,i$. By setting $t_n=\max\{\log(i+1)/\ell_{n,i}|1\le i\le k_n\}$, one has:
\begin{itemize}
\item[(1)] If $k_n=O(1)$, then there are $A>0$ and $\sigma_2>\sigma_1>0$ such that
\[
 1-\exp\{-e^{-a\sigma_2}\}\le d_{n,\textnormal{\tiny TV}}^{(c)}(at_n)\le 1-\exp\{-e^{-a\sigma_1}\},\quad\forall a>A,\,n\ge 1.
\]
In particular, $\mathcal{F}^\mathcal{P}_c$ has no pre-cutoff in the total variation.

\item[(2)] If $k_n\ra\infty$ and $\min\{\rho_{n,i}|i\ge m,n\ge 1\}\ge 4$ for $m$ large enough, then the following are equivalent.
\begin{itemize}
\item[(i)] $\mathcal{F}^\mathcal{P}_c$ has a total variation cutoff.

\item[(ii)] $\mathcal{F}^\mathcal{P}_c$ has a total variation pre-cutoff.

\item[(iii)] $t_n\ell_{n,1}\ra\infty$.

\item[(iv)] $T_{n,\textnormal{\tiny TV}}^{(c)}(\epsilon)\ell_{n,1}\ra\infty$ for all $\epsilon\in(0,1)$.

\item[(v)] $T_{n,\textnormal{\tiny TV}}^{(c)}(\epsilon)\ell_{n,1}\ra\infty$ for some $\epsilon\in(0,1)$.
\end{itemize}
Moreover, if $\mathcal{F}^\mathcal{P}_c$ has a total variation cutoff, then $T_{n,\textnormal{\tiny TV}}^{(c)}(\epsilon)\asymp t_n$ and $T_{n,\textnormal{\tiny TV}}^{(c)}(\epsilon)=T_{n,\textnormal{\tiny TV}}^{(c)}(\delta)+O(1/\ell_{n,1})$ for all $0<\epsilon<\delta<1$.
\end{itemize}

When $E_{n,i}$ contains $id$, (2) also holds under the replacement of $\mathcal{F}^\mathcal{P}_c$ and $T_{n,\textnormal{\tiny TV}}^{(c)}$ with $\mathcal{F}^\mathcal{P}$ and $T_{n,\textnormal{\tiny TV}}$.
\end{thm}

Theorem \ref{p-prodmoderate} is built on a list of theoretical results in Subsection \ref{ss-rots}. As its proof is a little complicated, we leave it to Subsection \ref{ss-dof}. In the following, we provide some remarks to comment the importance of Theorem \ref{p-prodmoderate}.

\begin{rem}\label{r-Peres}
{\rm (1)} Note that Theorem \ref{p-prodmoderate} also holds in the Hellinger distance due to (\ref{eq-tvhdcomp}) and Proposition \ref{p-tvhdcomp} and this is exactly what is done in the proof of the continuous time case. See Subsection \ref{ss-dof} for details.\\
{\rm (2)} In the proof of Theorem \ref{p-prodmoderate}, we obtain the order of the cutoff times, which is the same as $t_n$, but could not determine its asymptotic value, which relies on a more precise estimation of the convergence rate in the Helinger distance in Proposition \ref{p-moderatehd}.\\
{\rm (3)} A conjectured condition for the existence of cutoffs was introduced by Peres in 2004, which says that
\begin{equation}\label{eq-Peres}
 \text{A cutoff exists}\quad\Lra\quad\text{Mixing time}\times\text{Spectral gap}\ra\infty.
\end{equation}
In the setting of Theorem \ref{p-prodmoderate}, the spectral gaps of the $n$th random walks in $\mathcal{F}^\mathcal{P}$ and $\mathcal{F}^\mathcal{P}_c$ are of the same order as $\ell_{n,1}$ due to the assumption of $\inf\{Q_{n,i}(x)|x\in E_{n,i},\,1\le i\le k_n\,n\ge 1\}>0$. Consequently, the equivalence of (i) and (v) in (2) confirms the conjecture in (\ref{eq-Peres}) for products of random walks on finite groups with moderate growth.
\end{rem}

\subsection{Applications}\label{ss-tmr}

In this subsection, we apply Theorem \ref{p-prodmoderate} to the specific type of products introduced in \cite{CK16} and derive conditions on their cutoffs. Let $\mathcal{G}=(G_n,Q_n,U_n)_{n=1}^\infty$ be a family of random walks on finite groups driven by symmetric probabilities and $\mathcal{P}=(p_n)_{n=1}^\infty$ be a sequence of positive reals. Throughout this subsection, we write $\mathcal{G}^\mathcal{P}$ for the family, of which $n$th random walk refers to the product of $(G_i,Q_i,U_i)_{i=1}^n$ with respect to the probability vector $(p_1/q_n,...,p_n/q_n)$, where $q_n=p_1+\cdots+p_n$.
We now state the second main theorem of this article.

\begin{thm}\label{t-prodmoderate}
Consider the family $\mathcal{G}^\mathcal{P}$ introduced above. Let $E_n$ be the support of $Q_n$ and $\rho_n$ be the diameter of $(G_n,E_n)$. Assume that $G_n$ has $(A,d)$-moderate growth with respect to $E_n$, $\inf\{Q_n(x)|x\in E_n,\,n\ge 1\}>0$, $\rho_n\ge 4$ for $n$ large enough and $p_n/\rho_n^2\asymp \ell_n$ for some sequence $(\ell_n)_{n=1}^\infty$.
\begin{itemize}
\item[(1)] If $\ell_n\le\ell_{n+1}$ and $u_n=\max\{\log(i+1)/\ell_i|1\le i\le n\}$, then $\mathcal{G}^\mathcal{P}_c$ has a total variation cutoff if and only if $u_n\ra\infty$.

\item[(2)] If $\ell_n\ge\ell_{n+1}$ and $u_n=\max\{\log(i+1)/\ell_{n-i+1}|1\le i\le n\}$, then $\mathcal{G}^\mathcal{P}_c$ has a total variation cutoff if and only if $u_n\ell_n\ra\infty$.
\end{itemize}
In either case of (1) and (2), if $\mathcal{G}^\mathcal{P}_c$ has a total variation cutoff, then the cutoff time is of order $u_nq_n$. Further, if $E_n$ contains $id$ for all $n\ge 1$, then all above also holds for $\mathcal{G}^\mathcal{P}$.
\end{thm}

\begin{rem}
The lower bound of the graph diameter (at least $4$) in Theorem \ref{t-prodmoderate} is due to the requirement in Proposition \ref{p-moderatehd}. As the product of finitely many random walks has negligible contribution to the total variation (see e.g. Theorem \ref{p-prodmoderate}(1) for an illustration), one may suitably relax such a restriction on graph diameters as in Theorem \ref{t-prodmoderate}.
\end{rem}

\begin{proof}[Proof of Theorem \ref{t-prodmoderate}]
By Propositions \ref{p-dtctcomp} and \ref{p-tvhdcomp} in the next section, it suffices to prove this theorem for $\mathcal{G}^\mathcal{P}_c$ in the Hellinger distance. In the following, we will discuss (2), while (1) can be treated in a similar way.

Let $n_0>0$ be an integer such that $\rho_n\ge 4$ for $n\ge n_0$. For $n>n_0$, let $(G,Q,U)$ and $(G_n^{\mathcal{R}},Q_n^{\mathcal{R}},U_n^{\mathcal{R}})$ be products of $(G_i,Q_i,U_i)_{i=1}^{n_0}$ and $(G_i,Q_i,U_i)_{i=n_0+1}^n$ with respect to the probability vectors $(p_i/q)_{i=1}^{n_0}$ and $(p_i/q_n^{\mathcal{R}})_{i=n_0+1}^n$, where $q=p_1+\cdots+p_{n_0}$ and $q_n^{\mathcal{R}}=p_{n_0+1}+\cdots+p_n$. Clearly, the $n$th random walk in $\mathcal{G}^\mathcal{P}$ is the product of $(G,Q,U)$ and $(G_n^{\mathcal{R}},Q_n^{\mathcal{R}},U_n^{\mathcal{R}})$ with respect to the probability vector $(q/q_n,q_n^{\mathcal{R}}/q_n)$. Let $d_H$, $d_{n,H}^{\mathcal{R}}$ and $d_{n,H}^{(c)}$ be the Hellinger distances of the continuous time random walks associated with $(G,Q,U)$, $(G_n^{\mathcal{R}},Q_n^{\mathcal{R}},U_n^{\mathcal{R}})$ and the $n$th random walk in $\mathcal{G}^\mathcal{P}$. By (\ref{eq-prodhd}), one has
\begin{equation}\label{eq-dnh}
 d_{n,H}^{(c)}(t)^2=1-\left(1-d_H(qt/q_n)^2\right)
 \left(1-d_{n,H}^{\mathcal{R}}\left(q_n^\mathcal{R}t/q_n\right)^2\right).
\end{equation}

For the family $\mathcal{H}:=(G_n^{\mathcal{R}},Q_n^{\mathcal{R}},U_n^{\mathcal{R}})_{n=n_0+1}^\infty$ and the random walk $(G,Q,U)$, we set
\[
 v_n=q_n^{\mathcal{R}}\max_{1\le i\le n-n_0}\frac{\log(i+1)}{\ell_{n-i+1}},\quad v=q\max_{1\le i\le n_0}\frac{\log(i+1)}{\ell_{n_0-i+1}}.
\]
By Theorem \ref{p-prodmoderate}(1), there are constants $A>0$ and $\sigma_2>\sigma_1>0$ such that
\begin{equation}\label{eq-dh}
 1-\exp\left\{-e^{-a\sigma_2}\right\}\le d_H(av)\le 1-\exp\left\{-e^{-a\sigma_1}\right\},\quad\forall a>A,
\end{equation}
and, by Theorem \ref{p-prodmoderate}(2), $\mathcal{H}_c$ has a cutoff in the Hellinger distance if and only if $v_n\ell_n/q_n^\mathcal{R}\ra\infty$. Observe that, for $n>n_0$,
\[
 \frac{v_n}{q_n^\mathcal{R}}\le u_n\le\frac{v_n}{q_n^\mathcal{R}}+\max_{n-n_0<i\le n}\frac{\log(i+1)}{\ell_{n-i+1}}\le\frac{v_n}{q_n^\mathcal{R}}
 \left(1+\frac{\log(n+1)}{\log(n-n_0+1)}\right).
\]
This implies $v_n/q_n^\mathcal{R}\asymp u_n$.

We are now ready to derive (2). Assume that $u_n\ell_n\ra\infty$ or equivalently $v_n\ell_n/q_n^\mathcal{R}\ra\infty$. By Theorem \ref{p-prodmoderate}(2), $\mathcal{H}_c$ has a cutoff in the Hellinger distance and the cutoff time $w_n$ satisfies $w_n\asymp v_n$.
Immediately, this implies
\[
 \lim_{n\ra\infty}d_{n,H}^\mathcal{R}(aw_n)=\begin{cases}0&\text{for }a>1,\\1&\text{for }0<a<1,\end{cases}
 \quad\lim_{n\ra\infty}d_H\left(bw_n/q_n^\mathcal{R}\right)=0,\quad\forall b>0,
\]
where the second limit results from the second inequality of (\ref{eq-dh}) and the observation of $u_n\ge\log(n+1)/\ell_1\ra\infty$. Applying the above computations to (\ref{eq-dnh}) yields
\[
 \lim_{n\ra\infty}d_{n,H}^{(c)}\left(aq_nw_n/q_n^\mathcal{R}\right)
 =\begin{cases}0&\text{for }a>1,\\1&\text{for }0<a<1.\end{cases}
\]
This proves that $\mathcal{G}^\mathcal{P}_c$ has a cutoff in the Hellinger distance and the cutoff time is of order $u_nq_n$. Conversely, assume that $\mathcal{G}^\mathcal{P}_c$ has a cutoff in the Hellinger distance with cutoff time $w_n'$. Then, for $a>1$,
\[
 \lim_{n\ra\infty}d_{n,H}^\mathcal{R}\left(aq_n^\mathcal{R}w_n'/q_n\right)=0,\quad
 \lim_{n\ra\infty}d_H(aqw_n'/q_n)=0.
\]
By the first inequality of (\ref{eq-dh}), the latter limit implies $w_n'/q_n\ra\infty$, which yields $d_H(bw_n'/q_n)\ra 0$ for all $b>0$. In addition with the cutoff for $\mathcal{G}^\mathcal{P}_c$, we may derive from (\ref{eq-dnh}) that, for $0<a<1$,
\[
 1=\lim_{n\ra\infty}d_{n,H}^{(c)}(aw_n')=\lim_{n\ra\infty}
 d_{n,H}^\mathcal{R}\left(aq_n^\mathcal{R}w_n'/q_n\right).
\]
As a consequence, $\mathcal{H}_c$ has a cutoff in the Hellinger distance or equivalently $u_n\ell_n\asymp v_n\ell_n/q_n^\mathcal{R}\ra\infty$, as desired.
\end{proof}

\begin{rem}
Theorem \ref{t-prodmoderate}(1) can be in fact proved in a more direct way. Consider the exchange of the first random walk in $\mathcal{G}$ and the first random walk of which graph diameter is at least $4$, say the $N$th random walk. For the new family, all assumptions in Theorem \ref{p-prodmoderate} are fulfilled except the monotonicity of the sequence $\{\ell_N,\ell_2,...,\ell_{N-1},\ell_1,\ell_{N+1},...\}$. Such a concerning can be eliminated by using the original sequence $(\ell_n)_{n=1}^\infty$ along with a larger multiplicative constant and its reciprocal to bound the sequence $\{p_N/\rho_N^2,p_2/\rho_2^2,...,p_{N-1}/\rho_{N-1}^2,p_1/\rho_1^2,
p_{N+1}/\rho_{N+1}^2,...\}$. Under the above construction, Theorem \ref{t-prodmoderate} follows immediately from Theorem \ref{p-prodmoderate}(2).
\end{rem}

The following lemma is auxiliary to Theorem \ref{t-prodmoderate}, which provides conditions on the boundedness of $u_n$ and $u_n\ell_n$.

\begin{lem}\label{l-unln}
Let $\ell_n,u_n$ be constants in Theorem \ref{t-prodmoderate}.
\begin{itemize}
\item[(1)] If $\ell_n\le \ell_{n+1}$, then
\[
 u_n\ra\infty \quad\Lra\quad \sup_{n\ge 1}\frac{\log n}{\ell_n}=\infty.
\]

\item[(2)] Assume $\ell_n\ge\ell_{n+1}$. If $\ell_n/\ell_{n+1}\ra 1$, then $u_n\ell_n\ra\infty$. If $\liminf\limits_{n\ra\infty}\ell_n/\ell_{n+1}>1$, then $u_n\ell_n=O(1)$.
\end{itemize}
\end{lem}
\begin{proof}
(1) is obvious from the definition of $u_n$. For (2), we first consider the case $\ell_n/\ell_{n+1}\ra 1$. Note that, for $m\ge 1$,
\[
 \liminf_{n\ra\infty}u_n\ell_n\ge \liminf_{n\ra\infty}\frac{\log(m+1)}{\ell_{n-m+1}/\ell_n}=\log(m+1).
\]
Letting $m$ tend to infinity gives the desired limit. Next, we consider the case $\liminf\limits_{n\ra\infty}\ell_n/\ell_{n+1}>1$ and choose $N>0$ and $M>1$ such that $\ell_n/\ell_{n+1}\ge M$ for $n\ge N$. Immediately, this implies that $\ell_{n-m}/\ell_n\ge M^{m-N+1}$ for all $0\le m<n$ and $n\ge 1$. As a result, one has
\[
 u_n\ell_n=\max_{1\le i\le n}\frac{\log(i+1)}{\ell_{n-i+1}/\ell_n}\le M^N\sup_{i\ge 1}\frac{\log(i+1)}{M^i}<\infty.
\]
\end{proof}

The following corollary is a combination of Theorem \ref{t-prodmoderate} and Lemma \ref{l-unln}, of which proof is obvious and skipped.

\begin{cor}\label{c-prodmoderate}
Let $\mathcal{G}^\mathcal{P},E_n,\rho_n$ be as in Theorem \ref{t-prodmoderate}. Assume that $G_n$ has $(A,d)$-moderate growth with respect to $E_n$, $\inf\{Q_n(x)|x\in E_n,\,n\ge 1\}>0$, $\rho_n\ra\infty$ and $p_n\asymp \rho_n^2\ell_n$ for some monotonic sequence $(\ell_n)_{n=1}^\infty$.

\begin{itemize}
\item[(1)] If $\sup\{\log n/\ell_n|n\ge 1\}<\infty$ or $\liminf\limits_{n\ra\infty}\ell_n/\ell_{n+1}>1$, then $\mathcal{G}^\mathcal{P}_c$ has no total variation cutoff.

\item[(2)] If $\sup\{\log n/\ell_n|n\ge 1\}=\infty$ and $\lim\limits_{n\ra\infty}\ell_n/\ell_{n+1}=1$, then $\mathcal{G}^\mathcal{P}_c$ has a total variation cutoff.
\end{itemize}
In particular, if $\ell_n=\exp\{-n^\gamma\}$ with $\gamma>0$, then $\mathcal{G}^\mathcal{P}_c$ has a total variation cutoff if and only if $0<\gamma<1$. When $E_n$ contains $id$ for all $n\ge 1$, all above also holds for $\mathcal{G}^\mathcal{P}$.
\end{cor}

\subsection{Examples}\label{ss-ex}

In this subsection, we consider the randomized product in Proposition \ref{t-rp} for illustration of the results developed in Subsections \ref{ss-tmt}-\ref{ss-tmr}. Recall that $\mathcal{G}=(\mathbb{Z}_{n+2},Q_n,U_n)_{n=1}^\infty$, where $Q_n(0)=1/2$ and $Q_n(1)=Q_n(-1)=1/4$. It has been stated in the introduction that the diameter $\rho_n$ of $G_n$ w.r.t. $\{0,\pm 1\}$ is $\lfloor n/2+1\rfloor$ and $\mathbb{Z}_{n+2}$ has $(1,1)$-moderate growth w.r.t. $\{0,\pm 1\}$. As the randomness refers to the case that $\mathcal{P}=(p_n)_{n=1}^\infty$ is a sequence of positive random variables, we treat the specified cases separately in the following.

\begin{ex}[Polynomial random sequences]
Let $X_1,X_2,...$ be i.i.d. positive random variables, $\gamma>0$ and set $p_n=(X_1+\cdots+X_n)^\gamma$. Assume that the expectation $\mu$ of $X_1$ is finite. By the strong law of large numbers, one has
\[
 \frac{X_1+\cdots+X_n}{n}\sim\mu,\quad \text{almost surely},
\]
which implies
\[
 \frac{p_n}{\rho_n^2}\sim \ell_n:=\mu^\gamma n^{\gamma-2},\quad\text{almost surely}.
\]
Clearly, $(\ell_n)_{n=1}^\infty$ is monotonic, $\ell_n/\ell_{n+1}\ra 1$ and
\[
 \sup_{n\ge 1}\frac{\log n}{\ell_n}\begin{cases}<\infty&\text{for }\gamma>2,\\
 =\infty,&\text{for }0<\gamma\le 2.\end{cases}
\]
As a consequence of Corollary \ref{c-prodmoderate}, if $0<\gamma\le 2$, then $\mathcal{G}^\mathcal{P}$ has a total variation cutoff almost surely; if $\gamma>2$, then $\mathcal{G}^\mathcal{P}$ has no total variation cutoff with probability $1$.

\end{ex}

\begin{ex}[Exponential random sequences]
Let $Y_1,Y_2,...$ be i.i.d. positive random variables and $p_n=Y_1\times\cdots\times Y_n$. For $n\ge 1$, let $(\ell_{n,i})_{i=1}^n$ be a non-decreasing arrangement of $(p_i/\rho_i^2)_{i=1}^n$ and set
\[
 t_n=\sup_{1\le i\le n}\frac{\log(i+1)}{\ell_{n,i}}.
\]
Using a similar reasoning as in the proof of Theorem \ref{t-prodmoderate}, one may conclude using Theorem \ref{p-prodmoderate} that $\mathcal{G}^\mathcal{P}$ has a total variation cutoff if and only if $t_n\ell_{n,1}\ra\infty$.

To analyze the product $t_n\ell_{n,1}$, we assume that the expectation $\nu$ of $\log Y_1$ is finite. By the strong law of large numbers, there is an event $E$ with probability $1$ such that
\[
 \nu_n:=\frac{\log(p_n/\rho_n)}{n}\ra \nu\quad\text{on }E.
\]
In the following, we focus on the case $\nu>0$. By writing $p_n/\rho_n^2=e^{\nu_nn}$, one may select, for each $\omega\in E$, a constant $C(\omega)\in(0,1)$ such that $p_n(\omega)/\rho_n^2\ge C(\omega)e^{C(\omega)\nu n}$ for all $n\ge 1$. This implies that, on the event $E$,
\[
 \ell_{n,1}\asymp 1,\quad \ell_{n,i}\ge Ce^{C\nu i},\quad\forall 1\le i\le n,\,n\ge 1.
\]
Consequently, we obtain $t_n\ell_{n,1}=O(1)$ on $E$, which is equivalent to say that $\mathcal{G}^\mathcal{P}$ has no total variation cutoff with probability $1$.

The results in the above discussion are summarized in Proposition \ref{t-rp}.

\end{ex}

\section{Constructions of theoretical frameworks}\label{s-ctf}
This section is dedicated to proving Theorem \ref{p-prodmoderate}. In the first subsection, we review those required but developed results in the introduction. In the second subsection, we treat the discrete time and continuous time cases separately and provide proofs in detail.

\subsection{Review of technical supports}\label{ss-rots}

In this subsection, we survey those equivalences in Table \ref{tab-scheme} and state them by following the setting in the introduction. The first two propositions are supportive to the first two equivalences in Table \ref{tab-scheme} and, in fact, hold under more general assumptions.

\begin{prop}[Theorems 3.1 and 3.3 in \cite{CSal13-1}]\label{p-dtctcomp}
Let $\mathcal{F}=(G_n,Q_n,U_n)_{n=1}^\infty$ be a family of random walks on finite groups and $\delta=\inf_n Q_n(id)$. Assume that $\delta>0$ and, for some $\epsilon_0\in(0,1)$, $T_{n,\textnormal{\tiny TV}}(\epsilon_0)\ra\infty$ or $T_{n,\textnormal{\tiny TV}}^{(c)}(\epsilon_0)\ra\infty$. Then, in the total variation, $\mathcal{F}$ has a cutoff (resp. pre-cutoff) if and only if $\mathcal{F}_c$ has a cutoff (resp. pre-cutoff). Furthermore, if $\mathcal{F}$ or $\mathcal{F}_c$ presents a total variation cutoff, then $T_{n,\textnormal{\tiny TV}}(\epsilon)\sim T_{n,\textnormal{\tiny TV}}^{(c)}(\epsilon)$ for all $\epsilon\in(0,1)$ and, for sequences of positive reals, $(t_n)_{n=1}^\infty$ and $(b_n)_{n=1}^\infty$, satisfying $b_n=o(t_n)$,
\[
 \left|T_{n,\textnormal{\tiny TV}}(\epsilon)-t_n\right|=O(b_n),\,\forall \epsilon\in(0,1)\quad\Lra\quad \left|T_{n,\textnormal{\tiny TV}}^{(c)}(\epsilon)-t_n\right|=O(b_n),\,\forall \epsilon\in(0,1).
\]
\end{prop}

\begin{prop}\label{p-tvhdcomp}
Let $\mathcal{F}=(G_n,Q_n,U_n)_{n=1}^\infty$ be a family of random walks on finite groups and let $T_{n,\textnormal{\tiny TV}}^{(c)}$ and $T_{n,H}^{(c)}$ be the total variation and the Hellinger distance of the $n$th chain in $\mathcal{F}_c$. Then, $\mathcal{F}_c$ has a total variation cutoff (resp. pre-cutoff) if and only if $\mathcal{F}_c$ has a Hellinger distance cutoff (resp. pre-cutoff). Furthermore, if $\mathcal{F}_c$ presents a cutoff in either measurement, then $T_{n,\textnormal{\tiny TV}}^{(c)}(\epsilon)\sim T_{n,H}^{(c)}(\epsilon)$ for all $\epsilon\in(0,1)$.
\end{prop}
\begin{proof}
The equivalence of cutoffs is already discussed in Proposition 1.1 of \cite{CK16}, while the equivalence of pre-cutoffs is given by (\ref{eq-tvhdcomp}).
\end{proof}

To analyze products of random walks, we need a variant of Proposition \ref{p-moderate} in the Hellinger distance and, particularly, in the continuous time case.

\begin{prop}\label{p-moderatehd}
Let $(G,Q,U)$ be a symmetric random walk on a finite group and $d_{\textnormal{\tiny TV}}^{(c)},d_H^{(c)}$ be the total variation and the Hellinger distance of its associated continuous time random walk. If $G$ has $(A,d)$-moderate growth with respect to the support $E$ of $Q$, then there is $C>0$ depending only on $A,d$ such that
\[
 \frac{1}{2}e^{-Ct/\rho^2}\le d_{\textnormal{\tiny TV}}^{(c)}(t)\le Ce^{-\eta t/(2\rho^2)},\quad\forall t\ge 0,
\]
and
\[
 \frac{1}{4}e^{-Ct/\rho^2}\le d_{H}^{(c)}(t)\le Ce^{-\eta t/(4\rho^2)},\quad\forall t\ge 0,
\]
where $\rho$ is the diameter of $(G,E)$, $\eta=\min\{Q(x)|x\in E\}$ and both lower bounds require $\rho\ge 4$ in addition.
\end{prop}

\begin{rem}
{\rm (1)} Compared with Proposition \ref{p-moderate}, the generating set $E$ in Proposition \ref{p-moderatehd} need not contain $id$ and this means that the laziness of $(G,Q,U)$ is not required at all. In fact, the laziness of a continuous time walk can be seen from the identity $\frac{1}{2}(K-I)=\frac{1}{2}(K+I)-I$, where $K$ refers to the transition matrix determined by $Q$.\\
{\rm (2)} The prerequisite of the lower bound on the graph diameter (at least $4$) is due to the development of an upper bound on the spectral gap. See the proof of (3.2) in \cite{DS94} for details.
\end{rem}

\begin{proof}[Proof of Proposition \ref{p-moderatehd}]
First, we set $Q'=(Q+\mathbf{1}_{\{id\}})/2$ and $E'=E\cup\{id\}$. Let $K,K'$ be the transition matrices determined by $Q,Q'$, set $\rho,\rho'$ for the diameters of $(G,E),(G,E')$ and define $\eta=\min\{Q(x)|x\in E\}$ and $\eta'=\min\{Q'(x)|x\in E'\}$. Obviously, one has $K'=(K+I)/2$ and $H_t=H'_{2t}$, where $H_t=e^{-t(I-K)}$ and $H'_t=e^{-t(I-K')}$. Let $d'_{\textnormal{\tiny TV}},d_{\textnormal{\tiny TV}}^{(c)}$ be the total variations of $(G,Q',U),(G,H_t,U)$. By applying Proposition \ref{p-moderate} to $(G,Q',U)$, since $G$ has $(A,d)$-moderate growth with respect to $E$ (and, hence, with respect to $E'$), there is $C_2>0$ depending only on $A,d$ such that
\[
 d'_{\textnormal{\tiny TV}}(m)\le C_2e^{-\eta' m/(\rho')^2},\quad\forall m\ge 0.
\]
By the triangle inequality, this implies
\[
 d_{\textnormal{\tiny TV}}^{(c)}(t)\le e^{-2t}\sum_{m=0}^\infty\frac{(2t)^m}{m!}d'_{\textnormal{\tiny TV}}(m)
 \le C_2\exp\left\{2t\left(e^{-\eta'/(\rho')^2}-1\right)\right\}
 \le C_2e^{-\eta t/(2\rho^2)},
\]
where the last inequality comes from $\rho'\le \rho$, $\eta'\ge \eta/2$ and the fact that $e^{-u}\le 1-u/2$ for $u\in[0,1]$. To see a lower bound of the total variation, let $\lambda$ be the smallest nonzero eigenvalue of $I-K$. Note that $2d_{\textnormal{\tiny TV}}^{(c)}(t)=\|H_t-\Pi\|_{\infty\ra\infty}$, where $\Pi f:=\pi(f)\mathbf{1}$, $\|L\|_{\infty\ra\infty}:=\sup\{\|Lf\|_\infty:\|f\|_\infty\le 1\}$ and $\|f\|_\infty:=\max_x|f(x)|$. By the symmetry of $Q$, this implies
\[
 d_{\textnormal{\tiny TV}}^{(c)}(t)\ge \frac{1}{2}e^{-\lambda t},\quad\forall t\ge 0.
\]
Based on the $(A,d)$-moderate growth of $(G,E)$, Diaconis and Saloff-Coste showed in \cite[Equation (3.2)]{DS94} that if $\rho\ge 4$, then there is a constant $C_1>0$ depending only on $A,d$ such that $\lambda\le C_1/\rho^2$, where the assumption of $id\in E$ is in fact not required. This proves the desired bounds for the total variation with $C=\max\{C_1,C_2\}$, while the combination of (\ref{eq-tvhdcomp}) with such total variation bounds leads to bounds for the Hellinger distance.
\end{proof}

Finally, we introduce the fourth equivalence in Table \ref{tab-scheme}. Let $\mathcal{A}=\{a_{n,i}|1\le i\le k_n,n\ge 1\}$ and $\Lambda=\{\lambda_{n,i}|1\le i\le k_n,n\ge 1\}$ be triangular arrays of positive reals and set
\begin{equation}\label{eq-lt}
 \mathcal{F}(\mathcal{A},\Lambda)=(f_n)_{n=1}^\infty, \quad f_n(t)=\sum_{i=1}^{k_n}a_{n,i}e^{-\lambda_{n,i}t}.
\end{equation}
As $f_n$ is nonnegative and decreasing, we define the mixing time of $f_n$ by $T_n(\epsilon)=\min\{t\ge 0|f_n(t)\le\epsilon\}$ for $\epsilon>0$ and define the cutoff for $\mathcal{F}(\mathcal{A},\Lambda)$ as follows.

\begin{defn}\label{d-ltcutoff}
The family $\mathcal{F}(\mathcal{A},\Lambda)$ is said to present a cutoff if there is a sequence $(t_n)_{n=1}^\infty$ of positive reals such that
\[
 \lim_{n\ra\infty}f_n(at_n)=\begin{cases}0&\text{if }a>1,\\\infty&\text{if }0<a<1.\end{cases}
\]
In the above, $(t_n)_{n=1}^\infty$ or briefly $t_n$ is called a cutoff time.
\end{defn}

\begin{rem}
It is easy to check from Definition \ref{d-ltcutoff} that $\mathcal{F}(\mathcal{A},\Lambda)$ has a cutoff if and only if
$T_n(\epsilon)\sim T_n(\delta)$ for all $\epsilon>0$ and $\delta>0$. In particular,
if $\mathcal{F}(\mathcal{A},\Lambda)$ has a cutoff, then $T_n(\epsilon)$ is a cutoff time for all $\epsilon>0$.
\end{rem}

By expressing $f_n$ as a Laplace transform of some positive measure, the authors of \cite{CSal10} provided a criterion (Theorems 3.5 and 3.8 in \cite{CSal10}) to determine the cutoff for $\mathcal{F}(\mathcal{A},\Lambda)$. Later, such a method was refined in \cite[Theorem 2.4]{CHS16}. To see the details, we set, for $c>0$,
\begin{equation}\label{eq-lambdatau}
 \lambda_n(c)=\lambda_{n,j_n(c)},\quad \tau_n(c)=\max_{i\ge j_n(c)}\frac{\log(1+a_{n,1}+\cdots+a_{n,i})}{\lambda_{n,i}},
\end{equation}
where $j_n(c):=\min\{i\ge 1|a_{n,1}+\cdots+a_{n,i}>c\}$.

\begin{prop}[Theorem 2.4 in \cite{CHS16}]\label{p-ltcutoff}
Let $\mathcal{F}$ be the family in (\ref{eq-lt}), $T_n(\epsilon)$ be the mixing time of $f_n$ and $\lambda_n,\tau_n$ be the quantities in (\ref{eq-lambdatau}). Then, the following are equivalent.
\begin{itemize}
\item[(1)] $\mathcal{F}(\mathcal{A},\Lambda)$ has a cutoff.

\item[(2)] There is $\epsilon>0$ such that $T_n(\epsilon)\lambda_n(c)\ra\infty$ for all $c>0$.

\item[(3)] $\tau_n(c)\lambda_n(c)\ra\infty$ for all $c>0$.
\end{itemize}
In particular, $\tau_n(c)$ is a cutoff time for all $c>0$.
\end{prop}

\begin{rem}\label{r-ltcutoff}
It was shown in \cite[Lemma 2.5]{CHS16} that, if $\tau_n(c)\lambda_n(c)\ra\infty$, then $\tau_n(c')\lambda_n(c')\ra\infty$ for all $c'>c$.
\end{rem}

\subsection{Proof of Theorem \ref{p-prodmoderate}}\label{ss-dof}

The proof of Theorem \ref{p-prodmoderate} is based on some crucial techniques, of which proofs are either developed or involved and are addressed in the appendix for reference. See Lemmas \ref{l-prodhd}, \ref{l-dtctcomp} and \ref{l-hdsub} for details.

\begin{proof}[Proof of Theorem \ref{p-prodmoderate} (The continuous time case)]
To prove this proposition, it suffices to consider, by Proposition \ref{p-tvhdcomp} and (\ref{eq-tvhdcomp}), the Hellinger distance. Let $d_{n,H}^{(c)}$ and $d_{n,i,H}^{(c)}$ be the Hellinger distances of the $n$th and $(n,i)$th random walks in $\mathcal{F}^\mathcal{P}_c$ and $\mathcal{F}_c$, and set $\eta=\inf\{Q_{n,i}(x)|x\in E_{n,i},\,1\le i\le k_n\,n\ge 1\}$. By Proposition \ref{p-moderatehd}, there is $C_1>1$ such that, for all $1\le i\le k_n$ and $n\ge 1$,
\begin{equation}\label{eq-hddiam}
 \frac{1}{4}e^{-C_1t/\rho_{n,i}^2}\le d_{n,i,H}^{(c)}(t)\le C_1e^{-\eta t/(4\rho_{n,i}^2)}, \quad\forall t\ge 0,
\end{equation}
where $\rho_{n,i}\ge 4$ is required for the first inequality.

For (1), set $M=\sup_nk_n$. By (\ref{eq-hddiam}), one has
\[
 d_{n,i,H}^{(c)}(ap_{n,i}t_n/q_n)\le C_1e^{-a\eta\ell_{n,i}t_n/(4C)}\le
 C_1(i+1)^{-a\eta/(4C)}\le C_12^{-a\eta/(4C)},
\]
and, in addition with the fact $t_n\le \log(M+1)/\ell_{n,1}$,
\[
 d_{n,1,H}^{(c)}(ap_{n,1}t_n/q_n)\ge\frac{1}{4}e^{-aC_1C\ell_{n,1}t_n}\ge\frac{1}{4}(M+1)^{-aC_1C}.
\]
Consequently, the replacement of $A,p_i$ with $1/\sqrt{2},p_{n,i}/q_n$ in Lemma \ref{l-prodhd} yields
\[
 1-\exp\left\{-\frac{1}{16}(M+1)^{-2aC_1C}\right\}\le d_{n,H}^{(c)}(at_n)^2\le 1-\exp\left\{-MC_1^22^{1-a\eta/(2C)}\right\},
\]
for all $a>A:=(4C/\eta)(\log_2C_1+1/2)$ and $n\ge 1$. This proves (1).

For (2), note that (i)$\Ra$(ii) is clear from the definition of cutoffs and pre-cutoffs, and (iv)$\Ra$(v) is trivial. To prove the other equivalences, we first make some analysis on $d_{n,H}^{(c)}$. Let $C_1$ be the constant in (\ref{eq-hddiam}), $A$ be the constant defined as above and $N$ be a positive integer such that $\rho_{n,i}\ge 4$ for $i\ge N$ and $n\ge 1$. In a similar reasoning as before, one can show that
\begin{equation}\label{eq-prodhdlb}
 d_{n,H}^{(c)}(t)^2\ge 1-\exp\left\{-\sum_{i=1}^{k_n}d_{n,i,H}^{(c)}(p_{n,i}t/q_n)^2\right\},\quad\forall t>0,
\end{equation}
and
\begin{equation}\label{eq-prodhdub}
 d_{n,H}^{(c)}(t)^2\le 1-\exp\left\{-2\sum_{i=1}^{k_n}d_{n,i,H}^{(c)}(p_{n,i}t/q_n)^2\right\},\quad \forall t>At_n.
\end{equation}
By (\ref{eq-hddiam}), we have
\[
 \sum_{i=1}^{k_n}d_{n,i,H}^{(c)}(p_{n,i}t/q_n)^2\le C_1^2\sum_{i=1}^{k_n}e^{-\eta p_{n,i}t/(2q_n\rho_{n,i}^2)}\le C_1^2\sum_{i=1}^{k_n}e^{-\eta t\ell_{n,i}/(2C)}
\]
and
\[
 \sum_{i\in I_n}d_{n,i,H}^{(c)}(p_{n,i}t/q_n)^2\ge \frac{1}{16}\sum_{i\in I_n}e^{-2C_1p_{n,i}t/\rho_{n,i}^2}\ge \frac{1}{16}\sum_{i\in I_n}e^{-2CC_1t\ell_{n,i}},
\]
where $I_n=\{1\le i\le k_n|\rho_{n,i}\ge 4\}$. Putting the last terms in the above computations back to (\ref{eq-prodhdlb}) and (\ref{eq-prodhdub}) yields
\begin{equation}\label{eq-prodhdb}
 1-\exp\left\{-\frac{1}{16}g_n(2CC_1t)\right\}\le d_{n,H}^{(c)}(t)^2\le 1-\exp\left\{-2C_1^2f_n(\eta t/(2C))\right\},
\end{equation}
where $f_n(t)=\sum_{i=1}^{k_n}e^{-\ell_{n,i}t}$, $g_n(t)=\sum_{i\in I_n}e^{-\ell_{n,i}t}$ and the second inequality holds for $t>At_n$. We are now ready to proceed the proof of (ii)$\Ra$(iii)$\Ra$(iv) and (v)$\Ra$(i).

To see (ii)$\Ra$(iii), assume that $\mathcal{F}^\mathcal{P}_c$ presents a pre-cutoff in the Hellinger distance and let $s_n>0$ and $B_2>B_1>0$ be such that
\begin{equation}\label{eq-precutoff2}
 \lim_{n\ra\infty}d_{n,H}^{(c)}(B_2s_n)=0,\quad \liminf_{n\ra\infty}d_{n,H}^{(c)}(B_1s_n)=\alpha>0.
\end{equation}
Note that
\[
  f_n(at_n)\le\sum_{i=1}^{k_n}(i+1)^{-a}\le \int_1^\infty u^{-a}du=\frac{1}{a-1},\quad\forall a>1.
\]
By the second inequality of (\ref{eq-prodhdb}) and the fact $A>2C/\eta$, we have
\begin{equation}\label{eq-mixlb}
 d_{n,H}^{(c)}(at_n)^2\le 1-\exp\left\{-\frac{2C_1^2}{a\eta/(2C)-1}\right\},\quad\forall a>A,\,n\ge 1.
\end{equation}
Next, let's fix $a>(2C_1^2/\alpha^2+1)A$. By the fact of $A>2C/\eta$, one may derive $a\eta/(2C)-1>2C_1^2/\alpha^2$ and, by (\ref{eq-mixlb}), this leads to $d_{n,H}^{(c)}(at_n)^2\le 1-e^{-\alpha^2}<\alpha^2$.
As a result of the second limit in (\ref{eq-precutoff2}), we obtain that $at_n\ge B_1s_n$ for $n$ large enough.
In addition with the fact $1\in I_n$ for all $n\ge 1$, we may conclude from the first inequality of (\ref{eq-prodhdb}) that
\begin{align}
 0&=\lim_{n\ra\infty}d_{n,H}^{(c)}(B_2s_n)^2\ge
 \lim_{n\ra\infty}d_{n,H}^{(c)}(aB_2t_n/B_1)^2\notag\\
 &\ge 1-\exp\left\{-\frac{1}{16}\limsup_{n\ra\infty}
 e^{-2aB_2CC_1t_n\ell_{n,1}/B_1}\right\},\notag
\end{align}
which leads to (iii).

For (iii)$\Ra$(iv), assume that $t_n\ell_{n,1}\ra\infty$. By Proposition \ref{p-ltcutoff} and Remark \ref{r-ltcutoff}, the family $(f_n)_{n=1}^\infty$ has a cutoff with cutoff time $t_n$. By (\ref{eq-prodhdb}), this implies
\[
 \lim_{n\ra\infty}d_{n,H}^{(c)}(at_n)=0,\quad\forall a>A,\quad
 \lim_{n\ra\infty}d_{n,H}^{(c)}(at_n)=1,\quad\forall 0<a<1/(2CC_1),
\]
where the second limit also uses the fact
\[
 g_n(t)=f_n(t)-(f_n(t)-g_n(t))\ge f_n(t)-Ne^{-\ell_{n,1}t}.
\]
As a consequence, when $\epsilon\in(0,1)$, one has $t_n/(4CC_1)\le T_{n,H}^{(c)}(\epsilon)\le 2t_n$ for $n$ large enough, which gives (iv) and the order of the mixing time.

To show (v)$\Ra$(i), it suffices to prove $T_{n,H}^{(c)}(\epsilon)-T_{n,H}^{(c)}(\delta)=O(1/\ell_{n,1})$ for all $0<\epsilon<\delta<1$, which is exactly the specific conclusion in (2). First, we need a refinement of (\ref{eq-prodhdub}). Let $\delta\in(0,1)$. In Lemma \ref{l-prodhd}, the first inequality implies $T_{n,H}^{(c)}(\delta)\ge\max\{T_{n,i,H}^{(c)}(\delta)q_n/p_{n,i}|1\le i\le k_n\}$, while the third inequality yields
\[
 d_{n,H}^{(c)}(t)^2\le1-\exp\left\{-\frac{1}{1-\delta^2}\sum_{i=1}^{k_n}
 d_{n,i,H}^{(c)}\left(\frac{p_{n,i}t}{q_n}\right)^2\right\},\quad\forall t\ge T_{n,H}^{(c)}(\delta).
\]
Let $t=T_{n,H}^{(c)}(\delta)+a/\ell_{n,1}$ with $a>0$. By the quasi-submultiplicativity in Lemma \ref{l-hdsub}, one has
\[
 d_{n,i,H}^{(c)}\left(\frac{p_{n,i}t}{q_n}\right)\le 4d_{n,i,H}^{(c)}\left(\frac{p_{n,i}T_{n,H}^{(c)}(\delta)}{q_n}\right)d_{n,i,H}^{(c)}
 \left(\frac{ap_{n,i}}{\ell_{n,1}q_n}\right).
\]
Putting this back to the upper bound for $d_{n,H}^{(c)}(t)$ yields
\[
 d_{n,H}^{(c)}(t)^2\le 1-\exp\left\{-\frac{16}{1-\delta^2}\sum_{i=1}^{k_n}
 d_{n,i,H}^{(c)}\left(\frac{p_{n,i}T_{n,H}^{(c)}(\delta)}{q_n}\right)^2
 d_{n,i,H}^{(c)}\left(\frac{ap_{n,i}}{\ell_{n,1}q_n}\right)^2\right\}.
\]
Observe that the second inequality in (\ref{eq-hddiam}) and (\ref{eq-prodhdlb}) give
\[
 d_{n,i,H}^{(c)}\left(\frac{ap_{n,i}}{\ell_{n,1}q_n}\right)^2\le C_1^2e^{-a\eta/(2C)},\quad
 \sum_{i=1}^{k_n}d_{n,i,H}^{(c)}\left(\frac{p_{n,i}T_{n,H}^{(c)}(\delta)}{q_n}\right)^2
 \le\log\frac{1}{1-\delta^2}.
\]
Combining the last three inequalities leads to
\[
 d_{n,H}^{(c)}\left(T_{n,H}^{(c)}(\delta)+a/\ell_{n,1}\right)^2\le 1-\exp\left\{-\frac{C_1^2e^{-a\eta/(2C)}}{1-\delta^2}
 \log\frac{1}{1-\delta^2}\right\},\quad\forall a>0.
\]
As the right hand side tends to $0$ as $a$ tends to infinity, we obtain $T_{n,H}^{(c)}(\epsilon)-T_{n,H}^{(c)}(\delta)=O(1/\ell_{n,1})$ for $\epsilon\in(0,\delta)$, as desired. This finishes the proof of (2).
\end{proof}

\begin{proof}[Proof of Theorem \ref{p-prodmoderate} (The discrete time case)]
We shall prove the discrete time case by identifying the items in (2) with the continuous time case. First, we show that $T_{n,\textnormal{\tiny TV}}(\epsilon)\ra\infty$ and $T_{n,\textnormal{\tiny TV}}^{(c)}(\epsilon)\ra\infty$ for some $\epsilon\in(0,1)$. Let $d_{n,i,\textnormal{\tiny TV}}^{(c)}$ be the total variation of the $(n,i)$th random walk in $\mathcal{F}_c$ and let $N>0$ be a positive integer such that $\rho_{n,i}\ge 4$ for $i\ge N$ and $n\ge 1$. By Proposition \ref{p-moderatehd}, there is $C_2>0$ such that
\[
 d_{n,i,\textnormal{\tiny TV}}^{(c)}(t)\ge \frac{1}{2}e^{-C_2t/\rho_{n,i}^2}\ge\frac{1}{2}e^{-C_2t/16},\quad\forall t\ge 0,\,i\ge N,\,n\ge 1.
\]
Note that the first inequality in Lemma \ref{l-prodhd} also holds for the total variation (see \cite[Proposition 3.3]{CK16}) and this implies that, for $k_n>N$,
\begin{align}
 d_{n,\textnormal{\tiny TV}}^{(c)}(t)&\ge \max_{N\le i\le k_n}d_{n,i,\textnormal{\tiny TV}}^{(c)}\left(\frac{p_{n,i}t}{q_n}\right)\ge\frac{1}{2}\left\{-\frac{C_2t}{16}
 \min_{N\le i\le k_n}\frac{p_{n,i}}{q_n}\right\}\notag\\
 &\ge\frac{1}{2}\exp\left\{-\frac{C_2t}{16(k_n-N)}\sum_{i=N}^{k_n}\frac{p_{n,i}}{q_n}\right\}
 \ge\frac{1}{2}\exp\left\{-\frac{C_2t}{16(k_n-N)}\right\},\notag
\end{align}
where the last inequality uses the fact that $(p_{n,i}/q_n)_{i=1}^{k_n}$ is a probability vector.
As $k_n\ra\infty$, one has $T_{n,\textnormal{\tiny TV}}^{(c)}(1/4)\ge 8(k_n-N)/C_2$ for $n$ large enough, which yields $T_{n,\textnormal{\tiny TV}}^{(c)}(1/4)\ra\infty$. In the discrete time case, observe that, by the triangle inequality,
\[
 d_{n,\textnormal{\tiny TV}}^{(c)}(t)\le \sum_{m=0}^\infty e^{-t}\frac{t^m}{m!}d_{n,\textnormal{\tiny TV}}(m)\le \sum_{m=0}^\ell e^{-t}\frac{t^m}{m!}+\left(\sum_{m=\ell+1}^\infty e^{-t}\frac{t^m}{m!}\right)d_{n,\textnormal{\tiny TV}}(\ell).
\]
When $t=T_{n,\textnormal{\tiny TV}}^{(c)}(1/4)$ and $\ell=\lceil t/2\rceil$, it's easy to check (or to see from \cite[Lemma A.1]{CSal08}) that
\[
 \lim_{n\ra\infty}e^{-t}\sum_{m=0}^{\ell}\frac{t^m}{m!}=0,
\]
which leads to
\[
 \lim_{n\ra\infty}d_{n,\textnormal{\tiny TV}}\left(\left\lceil \frac{1}{2}T_{n,\textnormal{\tiny TV}}^{(c)}(1/4)\right\rceil\right)\ge\frac{1}{4}.
\]
Consequently, we obtain $T_{n,\textnormal{\tiny TV}}(1/5)\ge \frac{1}{2}T_{n,\textnormal{\tiny TV}}^{(c)}(1/4)$ for $n$ large enough and, thus, $T_{n,\textnormal{\tiny TV}}(1/5)\ra\infty$. Now, we are ready to prove (2) for the discrete time case.

Let ($*$)' with $*\in\{\text{i},\text{ii},\text{iv},\text{v}\}$ be respectively the corresponding statements for $\mathcal{F}^\mathcal{P}$ in Theorem \ref{p-prodmoderate}(2). Immediately, the equivalence of ($*$) and ($*$)' with $*\in\{\text{i},\text{ii}\}$ is given by Proposition \ref{p-dtctcomp}. Let $\mu_n$ and $\mu_{n,1}$ be the second largest eigenvalues of the transition matrices determined by $Q_n$ and $Q_{n,1}$. It is easy to check that $1-\mu_n\le(p_{n,1}/q_n)(1-\mu_{n,1})$. A similarly reasoning as in the proof of Proposition \ref{p-moderatehd} implies
\[
 d_{n,\textnormal{\tiny TV}}(m)\ge \mu_n^m,\quad 1-\mu_{n,1}\le C_3/\rho_{n,1}^2,\quad\forall m\ge 0,\,n\ge 1,
\]
where $C_3$ is a positive constant depending on $A,d$. As a result, this yields
\[
 d_{n,\textnormal{\tiny TV}}(m)\ge e^{-m(1-\mu_n)}\ge e^{-C_3mp_{n,1}/(q_n\rho_{n,1}^2)}
 \ge e^{-CC_3m\ell_{n,1}}.
\]
When $\mathcal{F}^\mathcal{P}$ has a total variation cutoff, one has
\[
 \exp\left\{-2CC_3\liminf_{n\ra\infty}T_{n,\textnormal{\tiny TV}}(\epsilon)\ell_{n,1}\right\}\le\lim_{n\ra\infty}d_{n,\textnormal{\tiny TV}}\left(2T_{n,\textnormal{\tiny TV}}(\epsilon)\right)=0,\quad\forall \epsilon\in(0,1).
\]
This proves (i)'$\Ra$(iv)'.

Based on the above discussions, it remains to show (v)'$\Ra$(v). Assume that $T_{n,\textnormal{\tiny TV}}(\epsilon_1)\ell_{n,1}\ra\infty$ for some $\epsilon_1\in(0,1)$. We will prove (v) by contradiction and thus assume the inverse that $T_{\xi_n,\textnormal{\tiny TV}}^{(c)}(\epsilon_2)\ell_{\xi_n,1}=O(1)$, where $\epsilon_2\in(0,1)$ and $(\xi_n)_{n=1}^\infty$ is an increasing sequence of positive integers. Note that we may restrict ourselves to the case of $\epsilon_2\le 1/4$. Set
$r_n=\sqrt{T_{\xi_n,\textnormal{\tiny TV}}(\epsilon_1)/\ell_{\xi_n,1}}$.
Obviously, $T_{\xi_n,\textnormal{\tiny TV}}^{(c)}(\epsilon_2)=o(r_n)$ and $r_n=o(T_{\xi_n,\textnormal{\tiny TV}}(\epsilon_1))$. By the quasi-submultiplicativity of the total variation, one has
\begin{align}
 \limsup_{n\ra\infty}2d_{\xi_n,\textnormal{\tiny TV}}^{(c)}(r_n)&\le \limsup_{n\ra\infty}\left(2d_{\xi_n,\textnormal{\tiny TV}}\left(T_{\xi_n,\textnormal{\tiny TV}}^{(c)}(\epsilon_2)\right)\right)^{\left\lfloor r_n\big/T_{\xi_n,\textnormal{\tiny TV}}^{(c)}(\epsilon_2)\right\rfloor}\notag\\
 &\le\limsup_{n\ra\infty}2^{-\left\lfloor r_n\big/T_{\xi_n,\textnormal{\tiny TV}}^{(c)}(\epsilon_2)\right\rfloor}=0.\notag
\end{align}
As a consequence of Lemma \ref{l-dtctcomp}(3), this implies
\[
 \lim_{n\ra\infty}d_{\xi_n,\textnormal{\tiny TV}}(\lceil ar_n\rceil)=0,\quad\forall a>1.
\]
But, however, as $ar_n=o(T_{\xi_n,\textnormal{\tiny TV}}(\epsilon_1))$, we have
\[
 \liminf_{n\ra\infty}d_{\xi_n,\textnormal{\tiny TV}}(\lceil ar_n\rceil)\ge\liminf_{n\ra\infty}d_{\xi_n,\textnormal{\tiny TV}}\left(T_{\xi_n,\textnormal{\tiny TV}}(\epsilon_1)-1\right)\ge\epsilon_1>0.
\]
This makes a contradiction and, hence, $T_{n,\textnormal{\tiny TV}}(\epsilon_2)\ell_{n,1}\ra\infty$, as desired.
\end{proof}

\section*{Acknowledgements}
The authors thank an anonymous referee for a careful reading of the manuscript and helpful comments.
The first author is partially supported by grant MOST 104-2115-M-009-013-MY3 and NCTS. The second author is partially supported by the Grant-in-Aid for Scientific Research (A)
25247007 and 17H01093.

\appendix

\section{Techniques and auxiliary results}

\begin{lem}\cite[Proposition 3.2]{CK16}\label{l-prodhd}
Let $(G_i,Q_i,U_i)_{i=1}^n$ be random walks on finite groups and $(G,Q,U)$ be their product according to the probability vector $(p_1,...,p_n)$. Let $d_{i,H}^{(c)}$ and $d_H^{(c)}$ be the Hellinger distances of the continuous time random walks associated with $(G_i,Q_i,U_i)$ and $(G,Q,U)$. Then, one has
\[
 d_H^{(c)}(t)\ge\max\left\{d_{i,H}^{(c)}(p_it)\Big|1\le i\le n\right\},
\]
and
\[
 1-\exp\left\{-\sum_{i=1}^nd_{i,H}^{(c)}(p_it)^2\right\}\le d_H^{(c)}(t)^2\le
 1-\exp\left\{-\frac{1}{1-A^2}\sum_{i=1}^nd_{i,H}^{(c)}(p_it)^2\right\},
\]
where the second inequality requires $t\ge \max\{T_{i,H}^{(c)}(A)/p_i|1\le i\le n\}$ with $A\in(0,1)$.
\end{lem}

\begin{lem}\cite[Proposition 3.1]{CSal13-1}\label{l-dtctcomp}
Let $\mathcal{F}=(G_n,Q_n,U_n)_{n=1}^\infty$ be a family of random walks on finite groups and $\mathcal{F}_c$ be the family of continuous time random walks associated with $\mathcal{F}$. For $n\ge 1$, let $d_{n,\textnormal{\tiny TV}},d_{n,\textnormal{\tiny TV}}^{(c)}$ be the total variations of the $n$th random walks in $\mathcal{F},\mathcal{F}_c$. Suppose $\inf_{n\ge 1}Q_n(id)>0$. Then, for any sequence $(t_n)_{n=1}^\infty$ tending to infinity,
\begin{itemize}
\item[(1)] $d_{n,\textnormal{\tiny TV}}(\lfloor t_n\rfloor)=1$ if and only if $d_{n,\textnormal{\tiny TV}}^{(c)}(t_n)=1$.

\item[(2)] If $d_{n,\textnormal{\tiny TV}}(\lceil t_n\rceil)=0$, then $d_{n,\textnormal{\tiny TV}}^{(c)}(at_n)\ra 0$ for all $a>1$.

\item[(3)] If $d_{n,\textnormal{\tiny TV}}^{(c)}(t_n)\ra 0$, then $d_{n,\textnormal{\tiny TV}}(\lceil at_n\rceil)\ra 0$ for all $a>1$.
\end{itemize}
\end{lem}

\begin{lem}\label{l-hdsub}
Consider an irreducible Markov chain on a finite or countable set $\mathcal{X}$ with transition matrix $K$ and stationary distribution $\pi$. Set $H_t=e^{-t(I-K)}$ and let $d_H,d_H^{(c)}$ be the maximum Hellinger distances defined by
\[
 d_H(m)=\sup_{x\in\mathcal{X}}\left(\frac{1}{2}\sum_{y\in\mathcal{X}}
 \left(\sqrt{K^m(x,y)}-\sqrt{\pi(y)}\right)^2\right)^{1/2},
\]
and
\[
 d_H^{(c)}(t)=\sup_{x\in\mathcal{X}}\left(\frac{1}{2}\sum_{y\in\mathcal{X}}
 \left(\sqrt{H_t(x,y)}-\sqrt{\pi(y)}\right)^2\right)^{1/2}.
\]
Then, the mappings
\[
 m\mapsto 4d_H(m),\quad t\mapsto 4d_H^{(c)}(t),
\]
are non-increasing and submultiplicative.
\end{lem}

\begin{proof}
We deal with the discrete time case, while the continuous time case can be shown in a similar way. Let $n,m$ be positive integers and $x,y\in\mathcal{X}$. Note that
\begin{equation}\label{eq-submulti}
\begin{aligned}
 \sqrt{K^{n+m}(x,y)}-\sqrt{\pi(y)}&=\frac{\sum_{z\in\mathcal{X}}[K^n(x,z)-\pi(z)][K^m(z,y)-\pi(y)]}
 {\sqrt{K^{n+m}(x,y)}+\sqrt{\pi(y)}}\\
 &=\sum_{z\in\mathcal{X}}A_zB_z,
\end{aligned}
\end{equation}
where
\[
 A_z=\left(\sqrt{K^n(x,z)}-\sqrt{\pi(z)}\right)
 \left(\sqrt{K^m(z,y)}-\sqrt{\pi(y)}\right)
\]
and
\[
 B_z=\frac{\left(\sqrt{K^n(x,z)}+\sqrt{\pi(z)}\right)
 \left(\sqrt{K^m(z,y)}+\sqrt{\pi(y)}\right)}
 {\sqrt{K^{n+m}(x,y)}+\sqrt{\pi(y)}}.
\]
Consider the following two cases.

{\it Case 1: $K^{n+m}(x,y)\ge \pi(y)$.} In this case, one may apply the Cauchy-Schwarz inequality to obtain
\[
 \left|\sqrt{K^{n+m}(x,y)}-\sqrt{\pi(y)}\right|^2\le\left(\sum_{z\in\mathcal{X}}A_z^2\right)
 \left(\sum_{z\in\mathcal{X}}B_z^2\right).
\]
Note that $(\sqrt{a}+\sqrt{b})^2\le 2(a+b)$ for all $a,b\ge 0$. As a result, this implies
\begin{align}
 \sum_zB_z^2
 &\le\frac{4\sum_{z\in\mathcal{X}}(K^n(x,z)+\pi(z))(K^m(z,y)+\pi(y))}{\left(\sqrt{K^{n+m}(x,y)}
 +\sqrt{\pi(y)}\right)^2}\notag\\
 &=\frac{4[K^{n+m}(x,y)+3\pi(y)]}{K^{n+m}(x,y)+\pi(y)+2\sqrt{K^{n+m}(x,y)\pi(y)}}\le 4,\notag
\end{align}
and, hence,
\[
 \left|\sqrt{K^{n+m}(x,y)}-\sqrt{\pi(y)}\right|^2\le4\left(\sum_{z\in\mathcal{X}}A_z^2\right).
\]

{\it Case 2: $K^{n+m}(x,y)<\pi(y)$.} For this case, set
\[
 E_1=\{z|\pi(z)>K^n(x,z),\,K^m(z,y)>\pi(y)\}
\]
and
\[
 E_2=\{z|\pi(z)<K^n(x,z),\,K^m(z,y)<\pi(y)\}.
\]
By (\ref{eq-submulti}), one has
\begin{align}
 0<\sqrt{\pi(y)}-\sqrt{K^{n+m}(x,y)}&\le\frac{\sum_{z\in E_1\cup E_2}[\pi(z)-K^n(x,z)][K^m(z,y)-\pi(y)]}{\sqrt{K^{n+m}(x,y)}+\sqrt{\pi(y)}}\notag\\
 &=\sum_{z\in E_1\cup E_2}A_zB_z.\notag
\end{align}
As before, we may apply the Cauchy-Schwarz inequality to get
\[
 \left|\sqrt{K^{n+m}(x,y)}-\sqrt{\pi(y)}\right|^2\notag\\
 \le\left(\sum_{z\in E_1\cup E_2}A_z^2\right)\left(\sum_{z\in E_1\cup E_2}B_z^2\right).
\]
Note that, for $z\in E_1$,
\begin{align}
 &\left(\sqrt{K^n(x,z)}+\sqrt{\pi(z)}\right)^2
 \left(\sqrt{K^m(z,y)}+\sqrt{\pi(y)}\right)^2\notag\\
 \le&4\left(\sqrt{K^n(x,z)}+\sqrt{\pi(z)}\right)^2K^m(z,y)=
 4\left(\sqrt{\frac{K^n(x,z)}{\pi(z)}}+1\right)^2\pi(z)K^m(z,y),\notag
\end{align}
and, for $z\in E_2$,
\begin{align}
 &\left(\sqrt{K^n(x,z)}+\sqrt{\pi(z)}\right)^2
 \left(\sqrt{K^m(z,y)}+\sqrt{\pi(y)}\right)^2\notag\\
 &\hskip1in\le 4K^n(x,z)\left(\sqrt{K^m(z,y)}+\sqrt{\pi(y)}\right)^2.\notag
\end{align}
By the Jensen inequality, this implies
\begin{equation}\label{eq-submulti1}
\begin{aligned}
 &\sum_{z\in E_1}\left(\sqrt{K^n(x,z)}+\sqrt{\pi(z)}\right)^2
 \left(\sqrt{K^m(z,y)}+\sqrt{\pi(y)}\right)^2\\
 &\hskip1in\le 4\left(\sqrt{\sum_{z\in E_1}K^n(x,z)K^m(z,y)}+\sqrt{c_1\pi(y)}\right)^2,
\end{aligned}
\end{equation}
where $c_1=\sum_{z\in E_1}\pi(z)K^m(z,y)/\pi(y)$, and
\begin{equation}\label{eq-submulti2}
\begin{aligned}
 &\sum_{z\in E_2}\left(\sqrt{K^n(x,z)}+\sqrt{\pi(z)}\right)^2
 \left(\sqrt{K^m(z,y)}+\sqrt{\pi(y)}\right)^2\\
 &\hskip1in\le 4\left(\sqrt{\sum_{z\in E_2}K^n(x,z)K^m(z,y)}+\sqrt{c_2\pi(y)}\right)^2,
\end{aligned}
\end{equation}
where $c_2=\sum_{z\in E_2}K^n(x,z)$. Summing up (\ref{eq-submulti1}) and (\ref{eq-submulti2}) yields
\[
 \sum_{z\in E_1\cup E_2}B_z^2\le\frac{4\left(\sqrt{\sum_{z\in E_1\cup E_2}K^n(x,z)K^m(z,y)}
 +\sqrt{(c_1+c_2)\pi(y)}\right)^2}
 {\left(\sqrt{K^{n+m}(x,y)}+\sqrt{\pi(y)}\right)^2}\le 8,
\]
while the last inequality uses the fact of $c_1\le 1$ and $c_2\le 1$. As a consequence, this leads to
\[
 \left|\sqrt{K^{n+m}(x,y)}-\sqrt{\pi(y)}\right|^2
 \le 8\left(\sum_{z\in E_1\cup E_2}A_z^2\right)\le 8\sum_{z\in\mathcal{X}}A_z^2,
\]
which is also applicable for Case 1.

Based on the result in the above discussions, we obtain
\begin{align}
 &4d_H(n+m)=\sup_{x\in\mathcal{X}}\left\{8\sum_{y\in\mathcal{X}}
 \left(\sqrt{K^{n+m}(x,y)}-\sqrt{\pi(y)}\right)^2\right\}^{1/2}\notag\\
 \le&\sup_{x\in\mathcal{X}}\left\{\sum_{y,z\in\mathcal{X}}
 8\left(\sqrt{K^n(x,z)}-\sqrt{\pi(z)}\right)^2\times
 8\left(\sqrt{K^m(z,y)}-\sqrt{\pi(y)}\right)^2\right\}^{1/2}\notag\\
 \le&\sup_{x\in \mathcal{X}}
 \left\{8\left(\sqrt{K^n(x,z)}-\sqrt{\pi(z)}\right)^2\right\}^{1/2}\times 4d_H(m)
 =4d_H(n)\times 4d_H(m).\notag
\end{align}
\end{proof}

\end{document}